\documentclass[a4paper,12pt]{article}

\newenvironment{proof}{\noindent {\bf Proof:}}{\hfill $\Box$}

\newtheorem{theorem}{Theorem}
\newtheorem{lemma}{Lemma}

\newtheorem{assumption}{Assumption}
\newtheorem{remark}{Remark}

\usepackage{booktabs}

\usepackage{algorithm}
\usepackage[noend]{algpseudocode}
\makeatletter
\def\BState{\State\hskip-\ALG@thistlm}
\makeatother
\usepackage{algorithmicx}

\usepackage{amsmath}
\usepackage{amssymb}
\usepackage{amsfonts}
\usepackage{graphicx}
\usepackage{ dsfont }

\usepackage[normalem]{ulem}
\usepackage{color}

\newcommand{\new}[1]{{\color{black}#1}}

\usepackage{tikz}
\usepackage{lipsum}

\newcommand{\bs}{\boldsymbol}
\newcommand{\mr}[1]{\mathrm{#1}}

\newcommand{\Mc}{\mathcal{M}}

\newcommand{\Rb}{\mathbb{R}}
\newcommand{\Nb}{\mathbb{N}}

\newcommand{\Xf}{\bs X}

\newcommand{\Ind}{\mathbb{I}}

\title{\bf Convex computation of extremal invariant measures of nonlinear dynamical systems and Markov processes}

\begin{document}

\author{Milan Korda$^{1,2}$, Didier Henrion$^{1,2,3}$, Igor Mezi{\'c}$^4$}

\footnotetext[1]{CNRS; LAAS; 7 avenue du colonel Roche, F-31400 Toulouse; France. {\tt korda@laas.fr, henrion@laas.fr}}
\footnotetext[2]{Faculty of Electrical Engineering, Czech Technical University in Prague,
Technick\'a 2, CZ-16626 Prague, Czech Republic.}
\footnotetext[3]{Universit\'e de Toulouse; LAAS; F-31400 Toulouse; France.}
\footnotetext[4]{University of California, Santa Barbara,\; {\tt mezic@engineering.ucsb.edu}}

\date{Draft of \today}

\maketitle

\begin{abstract}
We propose a convex-optimization-based framework for computation of invariant measures of polynomial dynamical systems and Markov processes, in discrete and continuous time. The set of all invariant measures is characterized as the feasible set of an infinite-dimensional linear program (LP). The objective functional of this LP is then used to single-out a specific measure (or a class of measures) extremal with respect to the selected functional such as physical measures, ergodic measures, atomic measures (corresponding to, e.g., periodic orbits) or measures absolutely continuous w.r.t. to a given measure. The infinite-dimensional LP is then approximated using a standard hierarchy of finite-dimensional semidefinite programming problems (SDPs), the solutions of which are truncated moment sequences, which are then used to reconstruct the measure. In particular, we show how to approximate the support of the measure as well as how to construct a sequence of weakly converging absolutely continuous approximations. \new{As a byproduct, we present a simple method to certify the non-existence of an invariant measure, which is an important question in the theory of Markov processes}. The presented framework, where a convex functional is minimized or maximized among all invariant measures, can be seen as a generalization of and a computational method to carry out the so called ergodic optimization, where linear functionals are optimized over the set of invariant measures. Finally, we also describe how the presented framework can be adapted to compute eigenmeasures of the Perron-Frobenius operator.
\end{abstract}

\begin{center}\small
\end{center}

\section{Introduction}
We propose a convex-optimization-based method for approximation of invariant measures. The method is based on the observation that the set of all invariant measures associated to a deterministic nonlinear dynamical system or a stochastic Markov process is given by the set of solutions to a \emph{linear} equation in the space of Borel measures. The problem of finding an invariant measure can therefore be formulated as the feasibility of an infinite-dimensional linear programming problem (LP). Adding an objective functional to this LP allows one to \emph{target} a particular invariant measure (or a class of invariant measures) such as the physical measure, ergodic measures, absolutely continuous measures, atomic measures etc. The formulation is flexible in the sense that whenever a variational characterization of a given class of measures is known, then it can be used within the proposed framework. The approach is functional analytic in nature, by and large devoid of geometric or topological considerations. The only underlying assumption is that the dynamics is polynomial. This assumption is made for computational convenience even though the approach is far more general, applicable to any algebra of functions closed under function composition (in discrete-time) or differentiation (in continuous-time).

The infinite-dimensional LP in the space of Borel measures is subsequently approximated along the lines of the classical \emph{Lasserre hierarchy}~\cite{lasserre2001global} using a sequence of finite-dimensional convex semidefinite programming problems (SDPs). The optimal values of the SDPs are proven to converge from below to the optimal value of the infinite-dimensional LP (for the analysis of the speed of convergence in a related setting of optimal control, see~\cite{korda2017convergence_opt}). The outcome of the SDP is an approximate truncated moment sequence of the invariant measure targeted; this sequence is proven to converge weakly to  the moment sequence of the target invariant measure, provided this measure is unique (otherwise every accumulation point of the sequence corresponds to an invariant measure). 

As a secondary contribution we describe a numerical procedure to approximate the support and density of the invariant measure using the truncated moment sequence obtained from the SDP. For the former, we provide confidence intervals enclosing, in the limit, a prescribed portion of the support; this is achieved using the Christoffel polynomial, an interesting object constructed from the Christoffel-Darboux kernel, which has already been utilized for support approximation in machine learning applications (e.g.,~\cite{lasserre2017empirical, pauwels2016nips}). For the latter, we construct a sequence of absolutely continuous measures with polynomial densities converging weakly to the target measure.

\new{An interesting by product of the approach presented is the possibility to certify the \emph{non-existence} of an invariant measure, which is particularly pertinent for Markov processes. Such certification boils down to proving the emptiness of a spectrahedron defining the feasible set of the SDPs solved.}

Finally, we also describe a generalization of the proposed approach to compute \emph{eigenmeasures} of the \emph{Perron-Frobenius} operator corresponding to a given (possibly complex) eigenvalue, with the invariant measures being a special case corresponding to eigenvalue one.


This work is a continuation of the movement to apply convex optimization-based techniques to nonconvex problems arising from dynamical systems theory and control. For example, the related problem of invariant set computation was addressed in~\cite{kordaMCI} whereas~\cite{lasserre2008nonlinear,gaitsgory2009linear,korda2016controller}
addressed optimal control; \cite{ozay2014convex}, \cite{ozay2015set} adressed model validation and switching system identification, respectively. The problem addressed here, i.e., invariant measure computation, was also addressed by this approach in~\cite{henrion2012kybernetika} in one spatial dimension; this work can therefore be seen as a generalization of~\cite{henrion2012kybernetika} to multiple dimensions and with a far more detailed theoretical and computational analysis. In the concurrent work \cite{MagronHenrionForets} the authors are also applying the Lasserre hierarchy for approximately computing invariant measures for polynomial dynamical systems, but there is no convex functional to be minimized and the focus is on distinguishing measures with different regularity properties (singular vs absolutely continuous).


Let us also mention the optimization-based approaches to invariant measure computation~\cite{bollt2005path,junge2017sighting}. These approaches are based on non-convex optimization and therefore have to deal with its inherent difficulties such as the existence of suboptimal local minimizers, saddle points or degeneracy. Therefore, contrary to the proposed convex-optimization based approach, these works do not provide convergence guarantees, despite being built on interesting ideas and showing promising practical performance.

The presented framework, where a convex user-specified functional is minimized among all invariant measures, can be seen as a generalization of and a computational method to carry out the so called \emph{ergodic optimization}~\cite{jenkinson2006ergodic,jenkinson2019,bochi2018}, where linear functionals are optimized among invariant measures (therefore leading to ergodic measures as the optimizers since these are the extreme points of the set of all invariant measures, hence the name ergodic optimization). 

\new{The presented approach based on optimization over Borel measures has a convex dual as an optimization over continuous functions that can be approximated by polynomial sum-of-squares. This line of research has been investigated independently for various problems from dynamical systems (e.g., \cite{chernyshenko,fantuzzi,goluskin}). Of particular relevance to this work is~\cite{tobasco} which is dual to our approach in the continuous-time setting and when the objective functional in our approach is restricted to be linear. A by-product of our work is therefore an asymptotic convergence guarantee for the bounds obtained by~\cite{tobasco}, provided that strong duality holds.
}

The paper is organized as follows. Section~\ref{sec:probState} formally states the problem of invariant measure computation. Section~\ref{sec:momHierarch} describes the moment hierarchies and applies them to the invariant measure computation problem. Section~\ref{sec:reconst} describes the reconstruction of the invariant measure from its moments. Section~\ref{sec:applications} discusses several concrete invariant measures to be targeted via the choice of objective functional to be optimized. Section~\ref{sec:contDet} describes an extension to continuous time systems and Section~\ref{sec:markov} to Markov processes (both in discrete and continuous time) \new{as well as discusses how to certify non-existence of invariant measures}. Section~\ref{sec:pf} extends the method to eigenmeasures of he Perron-Frobenius operator and Section~\ref{sec:numEx} presents numerical examples.

\section{Problem statement}\label{sec:probState}
For concreteness we present the approach for deterministic discrete-time dynamical systems. The case of stochastic Markov processes is treated in Section~\ref{sec:markov}; the continuous time cases are treated in Section~\ref{sec:contDet} and \ref{sec:contStoch}.

Consider therefore a deterministic discrete-time nonlinear dynamical system
\begin{equation}\label{eq:sys}
x^+ = T(x),
\end{equation}
where $x \in \mathbb{R}^n$ is the state, $x^+ \in \mathbb{R}^n$ is the successor state and each of the $n$ components of the mapping $T : \mathbb{R}^n \to \mathbb{R}^n$ is assumed to be a multivariate polynomial. 

An invariant measure for the dynamical system~(\ref{eq:sys}) is any nonnegative Borel measure $\mu$ satisfying the relation
\begin{equation}\label{eq:invarGen}
\mu(T^{-1}(\bf A)) = \mu(\bf A)
\end{equation}
for all Borel measurable $\bf A \subset \mathbb{R}^n$. In this paper we restrict our attention to invariant measures with support included in some compact set $\bs X \subset \Rb^n$. With this assumption, the relation~(\ref{eq:invarGen}) reduces to
\begin{equation}\label{eq:invar}
\int_{\bs X} f \circ T\,d\mu = \int_{\bs X} f \,d\mu
\end{equation}
for all $f \in C(\bs X)$.

When a measure $\mu$ is supported on a compact set $\bs X$, it follows from the Stone-Weierstrass Theorem that it is entirely characterized by its moment sequence $\bs y := (\bs y_\alpha)_{\alpha \in {\mathbb N}^n}
\in \Rb^\infty$, where
\begin{equation}\label{eq:mom_def}
\bs y_\alpha = \int_{\bs X} x^\alpha \, d\mu(x),
\end{equation}
with $x^\alpha:= x_1^{\alpha_1} \cdot\ldots\cdot x_n^{\alpha_n}$ and $\alpha \in \mathbb{N}^n$ running over all $n$-tuples of nonnegative integers and with $\Rb^\infty$ denoting the space of all real-valued sequences. In particular, for the choice $f(x)=x^{\alpha}$, relation (\ref{eq:invar}) becomes
\begin{equation}\label{eq:invarpoly}
\int_{\bs X} T^{\alpha}(x)\,d\mu(x) = \int_{\bs X} x^{\alpha} \,d\mu(x)
\end{equation}
for all $\alpha \in \mathbb{N}^n$, where $T^\alpha (x) := T_1(x)^{\alpha_1}\cdot\ldots\cdot T_n(x)^{\alpha_n}$. Since $T$ is a polynomial, (\ref{eq:invarpoly}) is a \emph{linear} constraint on the moments that can be written as
\begin{equation}\label{eq:invarmom}
A(\bs y) = 0,
\end{equation}
where $A : \Rb^\infty  \to \Rb^\infty$ is a linear operator.

We remark that (\ref{eq:invarmom}) characterizes all invariant measures associated to~(\ref{eq:sys}) with support in $\bs X$. In order to single out one invariant measure of interest we propose to use optimization. In particular, we propose to solve the infinite-dimensional convex optimization problem
\begin{equation}\label{opt:lp_inf}
\begin{array}{lll}
 \min\limits_{\bs y \in \bs M(\bs X)} & F(\bs y) \\
 \mathrm{s.t.} & A(\bs y) = 0\\
&  \boldsymbol y_0 = 1\\
\end{array}
\end{equation}
where the minimization is w.r.t. a sequence $\bs y$ belonging to the convex cone
\[
\bs M(\bs X):=\{\bs y \in \Rb^\infty \: :\: \bs y_\alpha = \int_{\bs X} x^\alpha d\mu(x), \: \alpha \in {\mathbb N}^n, \: \mu \in \Mc(\bs X)_+\}
\]
of moments of non-negative Borel measures on $\bs X$, the objective functional $F : \Rb^\infty \to \Rb $ is convex, and the constraint $\bs y_0 = 1$ is a normalization constraint enforcing that the measure is a probability measure.

We note that, since $T$ is polynomial and hence continuous and $\bs X$ is compact, the Krylov-Bogolyubov Theorem ensures that there exists at least one invariant measure for~(\ref{eq:sys}) and hence the optimization problem~(\ref{opt:lp_inf}) is always feasible.

\begin{remark}[Role of the objective function]
The role of $F$ is to \emph{target} or single out a specific invariant measure from the set of all invariant probability measures characterized by the constraints of~(\ref{opt:lp_inf}). In principle, $F$ can be any convex functional that  facilities this. In particular, it can be extended-valued (i.e. equal to $+\infty$), therefore encoding any constraints of interest such as $\mu$ being absolutely continuous or singular w.r.t. to a given measure. See Section~\ref{sec:applications} for concrete choices of $F$.
\end{remark}

A typical example encountered in practice for the choice of $F$ is
\begin{equation}\label{eq:ls}
F(\bs y) = \sum_{|\alpha| \le d} (\bs y_{\alpha} - \bs z_\alpha)^2,
\end{equation}
where $(\bs z_\alpha)_{|\alpha|\le d} $ is a given \emph{finite} vector of moments of total degree no more than $d$. The moments $\bs z_\alpha$ can be estimates of the first few moments of the invariant measure that we wish to compute obtained, e.g., from \new{observed} data or \new{by analytical reasoning (e.g., based on the symmetries of the problem)}. The optimization problem~(\ref{opt:lp_inf}) then seeks among all invariant measures $\mu$ the one which minimizes the discrepancy between the first moments $(\bs y_\alpha)_{|\alpha|\le d}$ of $\mu$ and the given moments $(\bs z_\alpha)_{|\alpha|\le d}$. Of course, the least-squares criterion in $(\ref{eq:ls})$ can be replaced by other convex metrics measuring the discrepancy between two truncated moment sequences.

\section{Moment hierarchy}\label{sec:momHierarch}
In this section we describe a hierarchy of finite dimensional convex optimization problems approximating the infinite-dimensional problem~(\ref{opt:lp_inf}) and prove that the solutions obtained from these approximations converge to a solution of~(\ref{opt:lp_inf}). For this we assume that the objective function $F(\bs y)$ in~(\ref{opt:lp_inf}) depends only on finitely many moments $(\bs y_\alpha)_{|\alpha| \le d}$.

The finite-dimensional approximations are derived from the so-called Lasserre hierarchy of approximations to the moment cone. In particular, we use semidefinite programming representable outer approximations to this cone and in addition we truncate the first equality constraint of~(\ref{opt:lp_inf}) by imposing the linear constraint (\ref{eq:invarpoly}) only for $f(x) = x^\alpha$, $|\alpha| \le k$, i.e., for all monomials of degree no more than $k$. By linearity of the constraint, this implies that the constraint is satisfied for all polynomials of degree no more than $k$. The degree $k$ is called relaxation degree.

Before writing down the finite-dimensional approximation of~(\ref{opt:lp_inf}), we first describe the construction of the finite-dimensional outer approximations to~$\bs M(\bs X)$.

\subsection{Finite-dimensional approximations of the moment cone}
Here we describe the semidefinite-programming representable outer approximation to $\bs M(\bs X)$. For this we assume that the compact set $\bs X$ is of the form\footnote{A set of the form~(\ref{eq:X}) is called basic semialgebraic; this class of sets is very rich, including balls, boxes, ellipsoids, discrete sets and various convex and non-convex shapes.}
\begin{equation}\label{eq:X}
\bs X := \{x \in \mathbb{R}^n \mid g_i(x) \ge 0, \;i=1,\ldots, n_g \}
\end{equation}
with $g_i$ being multivariate polynomials. Let us denote the unit polynomial by $g_0(x) := 1$.

The outer approximation $\bs M_d(\bs X)$ of degree $d$, $d$ even, is
 \begin{equation}\label{eq:Msup_def}
 \bs M_d(\bs X) := \{ \bs y \in \Rb^\infty \: :\:  M_d(g_i\: \bs y) \succeq 0, \;i=0,1,\ldots,n_g \},
 \end{equation}
 where $\cdot \succeq 0$ denotes positive semidefiniteness of a matrix and $M_d(g_i\:\bs y)$ are the so-called localizing moment matrices, to be defined below.
The convex cone $\bs M_d(\bs X)$ is an outer approximation to $\bs M(\bs X)$ in the sense that for any non-negative measure $\mu$ on $\bs X$ the moment vector $\bs y$ of $\mu$ belongs to $\bs M_d(\bs X)$.
%

The localizing moment matrices $M_d(g_i\:\bs y)$ are defined by \begin{equation}\label{eq:locMat}
M_d(g_i\:\bs y) = l_{\bs y}^d( g_i v_{d_i} v_{d_i}^\top),
 \end{equation}
where $d_i = \lfloor (d - \mr{deg}\,g_i )/ 2\rfloor$,
 \[
 v_d(x):=(x^\alpha)_{|\alpha|\leq d}
 \] 
and where the Riesz functional  $\ell^d_{\bs y}: \Rb[x]_d \to \Rb$ is defined for any $f  = \sum\limits_{|\alpha|\le d}\bs f_\alpha x^\alpha$ by
 \begin{equation}\label{eq:riesz}
l^d_{\bs y}(f) = \sum_{ |\alpha| \le d} \bs f_\alpha \bs y_\alpha.
\end{equation}
This functional mimicks integration with respect to a measure; in particular when $\bs y\in \Rb^{\binom{n+d}{d}}$ is a truncated moment vector of a measure $\mu$, then $l^d_{\bs y}(f) = \int f \, d\mu$ for any $f\in\Rb[x]_d$. In this case, the localizing matrices $M_d(g_i\:\bs y) $ are necessarily positive semidefinite, corresponding to the fact that $\int g_ip^2\,d\mu \ge 0$ for any polynomial $p$. Importantly, the following fundamental converse result states that if localizing moment matrices are positive semidefinite \emph{for all} $d \in \Nb$, then necessarily $\bs y$ is a moment vector of a nonnegative measure supported on $\bs X$.
In order for this to hold, the defining polynomials $g_i$ need  to satisfy the so-called Archimedian condition:
\begin{assumption}\label{as:arch}
There exist polynomials $\sigma_i$ and a constant $r\in\mathbb R$ such that 
\begin{equation}\label{eq:archCond}
r^2 - x^\top x = \sum_{i=0}^{n_g} \sigma_i g_i,
\end{equation}
where $\sigma_i = \sum_{j} p_{i,j}^2 $ with $p_{i,j}$ polynomial, i.e., each $\sigma_i$ is a sum of squares of other polynomials.
\end{assumption}
Assumption~\ref{as:arch} is an algebraic certificate of compactness of $\bs X$ because this assumption implies that $\bs X \subset \{x\mid x^\top x \le r^2\}$. Since $\bs X$ is assumed compact, this is a non-restrictive condition because a redundant constraint of the form $g_i = r^2 - x^\top x$ can always be added to the definition of $\bs X$ for a sufficiently large $r$ in which case Assumption~\ref{as:arch} is satisfied trivially.

\begin{theorem}[Putinar~\cite{putinar}]\label{thm:put}
Suppose that Assumption~\ref{as:arch} holds and that $\bs y \in {\bs M}_d(\bs X)$  for all $d \ge 0$.
Then there exists a unique non-negative measure $\mu$ on $\bs X$ such that (\ref{eq:mom_def}) holds for all $\alpha \in \mathbb{N}^n$.
\end{theorem}

 \subsection{Approximation of the infinite-dimensional convex problem}\label{sec:approx}
 Now we are ready to write down the finite-dimensional approximation to~(\ref{opt:lp_inf}). The first step in the approximation is to impose the equality constraint of~(\ref{opt:lp_inf}) only for all monomials of total degree no more than $k$ instead of for all continuous functions. That is, we impose,
 \[
 \int_{\bs X} T^\alpha(x) d\mu = \int_{\bs X} x^\alpha \,d\mu =0,\quad |\alpha | \le k,
 \]
where $T^\alpha (x) = T_1(x)^{\alpha_1}\cdot\ldots\cdot T_n(x)^{\alpha_n}$. This is a set of $k$ linear equations and since $T$ is polynomial it can be re-written in terms of the truncated moment sequence of the measure $\mu$ of degree no more than \[
d_k := k\, \mr{deg}\, T = k\max_{i=1,\ldots,n}\mr{deg}\,T_i.
\]
 In terms of the Riesz functional, this set of equalities becomes
 \begin{equation}\label{eq:eq_const}
  l_{\bs y}^{d_k}(T^\alpha(x)) = l_{\bs y}^{d_k}(x^\alpha),\quad |\alpha | \le k.
 \end{equation}
This set of equalities can be re-written in a matrix form as
\[
A_k \bs y = 0,
\]
for some matrix $A_k \in \Rb^{ \binom{n+k}{k} \times {\binom{n+d_k}{d_k}} }$
where $\bs y$ is the moment sequence of $\mu$.

The equality constraint $\int_{\bs X} 1\, d\mu $ translates to $\bs y_0 = 1$. 
The conic constraint $\bs y \in \bs M(\bs X)$ is replaced, according to the previous section, by the constraint $\bs y \in  \mathcal{M}_{d_k}(\bs X)$.
This leads to the following finite-dimensional relaxation of order $k$ of the infinite-dimensional problem~(\ref{opt:lp_inf})
\begin{equation}\label{opt:relax}
\begin{array}{lll}
 \min\limits_{\bs y \in \Rb^{\binom{n+d_k}{d_k}}} & F(\bs y) \\
 \mathrm{s.t.} &  A \bs y = 0 \vspace{0.2mm}\\
&  \bs y_0 = 1 \vspace{0.2mm}\\
&M_{d_k}(g_i \: \bs y) \succeq 0, \quad \forall\, i=0,1,\ldots,n_g.
\end{array}
\end{equation}
In optimization problem~(\ref{opt:lp_inf}), a convex function is minimized over a convex semidefinite-programming representable set and hence~(\ref{opt:lp_inf}) is a convex optimization problem. Provided that the objective functional $F(\bs y)$ is also semidefinite programming representable (e.g., it is of the form (\ref{eq:ls})), then the problem (\ref{opt:lp_inf}) is a semidefinite programming problem and hence can be readily solved by off-the-shelf software (e.g., MOSEK or SeDuMi~\cite{sedumi}). Importantly, the finite-dimensional relaxation~(\ref{opt:relax}) can be derived from the abstract form~(\ref{opt:lp_inf}) and passed to a selected SDP solver automatically with the help of the modelling software Gloptipoly 3~\cite{gloptipoly3} and Yalmip~\cite{yalmip}.

An immediate observation is that problem~(\ref{opt:relax}) is a relaxation of problem~(\ref{opt:lp_inf}) in the sense that the moment sequence of any measure feasible in~(\ref{opt:lp_inf}) truncated up to degree $d_k$ is feasible in~(\ref{opt:relax}). Therefore in particular for any $k$, the optimal value of~(\ref{opt:relax}) provides a lower bound on the optimal value of~(\ref{opt:lp_inf}). In the following section we study the convergence of these lower bounds to the optimal value of~(\ref{opt:lp_inf}) as as well as convergence of the minimizers of~(\ref{opt:relax}) to a minimizer of~(\ref{opt:lp_inf}).

\subsection{Convergence of approximations}

In this section we prove convergence of the finite-dimensional approximations~(\ref{opt:relax}) to a solution to the infinite-dimensional optimization problem~(\ref{opt:lp_inf}).

\begin{theorem}\label{thm:conv}
Suppose that Assumption~\ref{as:arch} holds, that the function $F$ is lower semi-continuous\footnote{More precisely, $F$ is assumed to be lower semi-continuous with respect to the product topology on the space of sequences $\Rb^\infty$. This is in particular satisfied if $F$ depends only on finitely many moments as, for example, in~(\ref{eq:ls}).} and let $\bs y^k \in \Rb^{\binom{n+d_k}{d_k}}$ denote an optimal solution to~(\ref{opt:relax}) and $p^\star$ the optimal value of~(\ref{opt:lp_inf}). Then the following holds:
\begin{enumerate}
\item $\lim_{k\to\infty} F(\bs y^k) = p^\star$
\item There exists a subsequence $(k_i)_{i=1}^\infty$ such that $\bs y^{k_i}$ converges pointwise to a moment sequence of an invariant measure $\mu^\star$ attaining the minimum in~(\ref{opt:lp_inf}).
\item In particular, if there is a unique invariant measure $\mu^\star$ attaining the minimum in~(\ref{opt:lp_inf}), then $\bs y^k$ converges to the moment sequence of $\mu^\star$.
\end{enumerate}
\end{theorem}
\begin{proof}

The proof follows a standard argument (see, e.g., \cite{lasserreBook}). By Assumption~\ref{as:arch}, for every $\bar \alpha \ge 0$ there exists a $k_0$ such that $|\bs y_{\alpha}| \le M^{|\alpha|}$ for any vector $\bs y$ satisfying the constraints of~(\ref{opt:relax}) for $k\ge k_0$ and for any $\alpha \in \mathbb{N}^n$ satisfying $|\alpha|\le \bar \alpha$, where $M$ is the constant from~(\ref{eq:archCond}). This statement implies that each component of $\bs y^k$ is bounded for sufficiently large $k$. To see this, let $f = \sigma_0+\sum_{i}\sigma_i g_i$ be the right-hand-side polynomial from~(\ref{eq:archCond}) and let $k_0$ be the smallest number $d \ge 1$ such that $d \:\mr{deg}\, T \ge \mr{deg}\,\sigma_0/2$ and $d \:\mr{deg}\, T \ge \lfloor(\sigma_i-\mr{deg}\,g_i)/2\rfloor$. Then necessarily $\ell_{\bs y}^{d_k}(\sigma_0) \ge 0$ and $\ell_{\bs y}^{d_k}(\sigma_i g_i) \ge 0$ for any $k \ge k_0$. Therefore by linearity we have $\ell_{\bs y}^{d_k}(f) \ge 0$. But since $f = r^2-x^\top x$ by~(\ref{eq:archCond}) we get $ \ell_{\bs y}^{d_k}(r^2) \ge \ell_{\bs y}^{d_k}(x^\top x) = \sum_{i=1}^n \ell_{\bs y}^{d_k}({x_i}^2) $. Since $\ell_{\bs y}^{d_k}(r^2) = r^2\ell_{\bs y}^{d_k}(1) = r^2\bs y_0$ and $\bs y_0 = 1$ by the second constraint of~(\ref{opt:relax}) and since $\ell_{\bs y}^{d_k}({x_i}^2) \ge 0$, we conclude that $\ell_{\bs y}^{d_k}({x_i}^2) \in [0,r^2]  $. Proceeding recursively, applying the same reasoning to $x^\alpha f$ with $\alpha_i$ even for all $\alpha\in \mathbb{N}^n$ satisfying $|\alpha| \le \bar{\alpha} := \max\{\alpha \in\mathbb{N}^n \mid  |\alpha|/2 + (\mr{deg}\,\sigma_i)/2 \le \lfloor (d_k - \mr{deg}\, g)/2 \rfloor\}$, we conclude that all even moments $\bs y_{\alpha}$, $|\alpha| \le \bar{\alpha}$ lie in $[0,r^{2|\alpha|}]$. Since even moments are on the diagonal of the matrix $M_d(\bs y) \succeq 0$ and since the off-diagonal elements of a positive semidefinite matrix are bounded in magnitude by the diagonal elements, the conclusion follows.

Having established that $\limsup_{k \to\infty} |\bs y^k_\alpha| <\infty$ for each $\alpha \in \mathbb{N}^n$, it follows using a standard diagonal argument that we can extract a subsequence $\bs y^{k_i}$ satisfying, for each $\alpha\in \Nb^n$, $\lim_{i\to \infty}\bs y^{k_i}_\alpha = \bs y^\star_\alpha $ with $\bs y^\star_\alpha \in \Rb$. To conclude the proof it remains to show that $\bs y^\star$ is a moment sequence of a measure attaining the minimum in~(\ref{opt:lp_inf}). Using Theorem~\ref{thm:put}, it follows that $\bs y^\star$ is a moment sequence of a non-negative measure $\mu^\star$ on $\bs X$ since $M_d(g_i\:\bs y^\star) \succeq 0$ by continuity of the mapping $M \mapsto \lambda_{\mr{min}}(M)$, where $\lambda_{\mr{min}}(M)$ denotes the minimum eigenvalue of a symmetric matrix (or equivalently by closedness of the cone of positive semidefinite matrices). In addition, $\bs y ^\star$ satisfies the equality constraints of~(\ref{opt:lp_inf}) by continuity since each row of the matrix $A$ has only finitely many non-zero elements. Therefore $\mu^\star$ is an invariant measure. Finally, 
since (\ref{opt:relax}) is a relaxation of~(\ref{opt:lp_inf}) we have $F(\bs y^{k_i}) \le p^\star $ . By the lower semi-continuity of $F$ we also have $F(\bs y^\star) \le \lim_{i\to\infty}F(\bs y^{k_i}) \le p^\star$ and hence necessarily $F(\bs y^\star) = p^\star$ since $\mu^\star$ is feasible in~(\ref{opt:lp_inf}) and therefore  $F(\bs y^\star) \ge p^\star$.
\end{proof}

\section{Reconstruction of measure from moments}\label{sec:reconst}
In this section we show how the solutions to the finite dimensional relaxations~(\ref{opt:relax}) in the form of a truncated moment sequence can be used to approximately reconstruct the invariant measure. In particular we show how to approximate the support of the measure and how to construct a sequence of absolutely continuous measures converging weakly to the invariant measure. In this section, we assume that the optimal invariant measure is unique, which holds generically (in the Baire category sense) by~\cite[Theorem 3.2]{jenkinson2006ergodic}.

\begin{assumption}[Unique invariant measure]\label{unique}
Convex problem (\ref{opt:lp_inf}) has a unique solution denoted by $\mu^\star$.
\end{assumption}

\subsection{Approximation of the support}
In this section we show how the solutions to the finite dimensional relaxations~(\ref{opt:relax}) can be used to approximate the support of the invariant measure $\mu^\star$.  \new{The approximations constructed here aim at enclosing a certain prescribed portion of the support. Guaranteed outer approximates to the global attractor (on which certain invariant measures are supported) can be computed using the approach of~\cite{schlosser_attractor}}.

In order to construct the approximations we utilize a certain polynomial constructed for a vector of moments of a given measure. Assume that we are given the sequence of moments $\bs y$ of a non-negative measure $\mu$ on $\bs X$. Then from the truncated moments of degree up to $d$, $d$ even, we define the Christoffel polynomial
\begin{equation}\label{eq:Qdef}
q_{d}^{\bs y}(x) = v_{d/2}(x)^\top M_{d}(\bs y)^{-1} v_{d/2}(x),
\end{equation}
where $v_{d/2}(\cdot)$ is the basis vector of all monomials up to degree $d/2$ with the same ordering as the vector of moments $\bs y$.  The polynomial  $q_{d}^{\bs y}$ is well defined as long the moment matrix $M_d(\bs y)$ is invertible, which is satisfied if and only if $\mu$ is not supported on the zero level set of a polynomial of degree $d/2$ or less.

The sublevel sets of the polynomial $q_{d}^{\bs y}(x)$ have a remarkable property of approximating the shape of the support of the measure $\mu$. In the real multivariate domain, this was observed recently for empirical measures (sums of Dirac masses) in~\cite{pauwels2016nips} and subsequently studied analytically for measures with certain regularity properties in~\cite{lasserre2017empirical}. In the complex domain, the theory is far more developed; see, e.g.,~\cite{gustafsson2009bergman} and references therein.

Here we use the following simple result which holds for arbitrary probability measures $\mu$.
\begin{lemma}\label{lem:confBound}
Let $\mu$  be a probability measure on $\bs X$ with moment sequence $\bs y$, let $ \epsilon\in [0,1)$ be given and let $\gamma_\epsilon = \frac{1}{1-\epsilon}\binom{n+d/2}{n} $ and assume that $M_d(\bs y)$ is invertible. Then
\begin{equation}\label{eq:confReg}
\mu(\{x : q_d^{\bs y}(x) \le \gamma_\epsilon\}) \ge \epsilon.
\end{equation}
\end{lemma}
\begin{proof}
We bound the complementary event:
\[
\mu(\{x : q_d^{\bs y}(x) > \gamma_\epsilon\})  = \int_{\{x : q_d^{\bs y}(x)  > \gamma_\epsilon  \}} 1\, d\mu = \int_{\{x : \gamma_\epsilon^{-1}q_d^{\bs y}(x)  > 1  \}} 1\, d\mu \le \gamma_\epsilon^{-1}\int_{\Xf}\, q_d^{\bs y}(x)\, d\mu(x),
\]
where we have used the fact that $q_d^{\bs y}$ is nonnegative. Using the definition of $q_d^{\bs y}$ we get
\begin{align*}
\int_{\Xf}\, q_d^{\bs y}(x)\, d\mu(x) &= \int_{\Xf} v_{d/2}(x)^\top M_d(\bs y)^{-1}v_{d/2}(x) \, d\mu(x) = \mathop{\mr{trace}}\Big\{M_d(\bs y)^{-1}  \int_{\Xf} v_{d/2}(x)v_{d/2}(x)^\top \, d\mu(x) \Big\} \\
& = \mathop{\mr{trace}}\{ M_d(\bs y)^{-1}  M_d(\bs y)\} = \mathop{\mr{trace}}\Big\{ \Ind_{\binom{n+d/2}{n}}\Big\} = \binom{n+d/2}{n},
\end{align*}
where $\Ind_k$ denotes the identity matrix of size $k$. Therefore
\[
\mu (\{x : q_d^{\bs y}(x) > \gamma_\epsilon \big\} \le \frac{1}{\gamma_\epsilon}\binom{n+d/2}{n}
\]
and the result follows since $\mu(\{x : q_d^{\bs y}(x) \le \gamma_\epsilon\})$ = $1- \mu(\{x : q_d^{\bs y}(x) > \gamma_\epsilon\})$.
\end{proof}

Lemma~\ref{lem:confBound} can be readily used to construct approximations to the support of $\mu$ in the form
\[
 \big\{x : q_d^{\bs y^k}(x) \le \gamma_\epsilon \big\}
\]
where  $\bs y^k \in \Rb^{\binom{n+d_k}{d_k}}$ is a solution to the $k$th order relaxation~(\ref{opt:relax}). We remark that $\bs y^k$ only approximates the moments of $\mu^\star$, with guaranteed convergence by Theorem~\ref{thm:conv}, and hence the bound from Lemma~\ref{lem:confBound} may be violated for finite $k$.

\subsection{Weakly converging approximations}
In this section we show how the to construct a sequence of absolutely continuous measures (w.r.t. the Lebesgue measure) converging weakly to the invariant measure. This is especially useful if in fact the invariant measure possesses a density with respect to the Lebesgue measure although the approach is general and always provides a sequence of signed measures with polynomial densities that converges weakly to the invariant measure. The idea is simple: given a vector $\bs y\in \Rb^{\binom{n+d}{n}}$ (e.g., a truncated moment vector of a measure), we can always represent the Riesz functional $\ell_{\bs y}^d:\Rb[x]_d \to\Rb$ as
\begin{equation}\label{eq:rho}
\ell_{\bs y}^d(p) = \int_{\bs X} p(x) q(x)dx
\end{equation}
for some polynomial $q \in \Rb[x]_d$, provided that the set $\Xf$ has a nonempty interior. Indeed, by linearity it suffices to satisfy~(\ref{eq:rho}) for $p(x) = x^\alpha$, $|\alpha| \le d$, which leads to a system of linear equations
\begin{equation}\label{eq:rho2}
 M^L_d \bs q =\bs y,
\end{equation}
where $\bs q$ is the coefficient vector of polynomial $q$ in the monomial basis with the same ordering as the vector of moments $\bs y$, and $M^L_d$ is the moment matrix of the Lebesgue measure on $\bs X$ of degree $2d$, i.e.,
\[
M^L_d := \int_{\bs X} v_d(x) v_d(x)^\top \, dx.
\]
Provided that the interior of $\Xf$ is nonempty, matrix $M^L_d$ is invertible and hence the linear system of equations~(\ref{eq:rho2}) has a unique solution $\bs q = (M^L_d)^{-1}\bs y$.
This approach applied to the solutions of~(\ref{opt:relax}) leads to the following result:
\begin{theorem}\label{thm:weakDense}
Let Assumption \ref{unique} hold, suppose that the interior of $\bs X$ is nonempty, denote $\bs y^k \in \Rb^{\binom{n+d_k}{n}}$ any solution to~(\ref{opt:relax}), and let
\begin{equation}\label{eq:crho}
\bs q^k := (M^L_{d_k})^{-1} \bs  y^k.
\end{equation}
Then the signed measures with densities $q_k(x) := v^\top_k(x) \bs q^k \in \Rb[x]_k$ with respect to the Lebesgue measure, converge weakly star on $\bs X$ to the invariant measure $\mu^\star$.
\end{theorem}
\begin{proof}
Verifying weak star convergence means that $\lim_{k\to\infty} \int_{\Xf} f q_k\,dx  = \int_{\Xf} f \, d\mu^\star$ for all $f \in C(\Xf)$. Since $\Xf$ is compact, it is enough to verify this relationship for all $f$ of the form $f = x^\alpha$, $\alpha \in \Nb^n$, which forms a basis of the space of all polynomials, which is a dense subspace of $C(\Xf)$. By construction we have $\int_{\bs X} x^\alpha q_k = \bs y_\alpha^k$ for $k\ge |\alpha|$ and hence
\[
\lim_{k\to\infty} \int_{\Xf} x^\alpha q_k(x)\,dx = \lim_{k\to\infty} \bs y^k_\alpha = \bs y_\alpha^\star = \int_{\Xf} x^\alpha \, d\mu^\star(x)
\]
by Theorem~\ref{thm:conv}, part 3.
\end{proof}

Theorem~\ref{thm:weakDense} says that the density approximations constructed from the solutions to~(\ref{opt:relax}) using~(\ref{eq:crho}) converge in the weak star topology to the invariant measure optimal in~(\ref{opt:lp_inf}). From a practical point of view, e.g., for the purpose of visualization, we recommend using polynomial densities with coefficients
\begin{equation}\label{eq:rhobar}
(M^L_k)^{-1} \bar{\bs   y}^k,
\end{equation}
where $\bar{\bs   y}^k \in \Rb^{\binom{n+k}{n}}$ is the vector of the first $\binom{n+k}{n}$ elements of $\bs y^k$, rather than the full vector $\bs y^k \in \Rb^{\binom{n+d_k}{n}}$. This is because the invariance constraint~(\ref{eq:eq_const}) is imposed only for all monomials up to degree $k$ and hence moments of degrees higher than $k$ are less constrained in~(\ref{opt:relax}) and hence are likely to be less accurate approximations to the true moments.

\section{Applications - choosing the objective function}\label{sec:applications}
In this section we list a several classes of measures that can be targeted through the choice of the objective function $F$ of convex problem~(\ref{opt:lp_inf}).

\subsection{Physical measures}\label{sec:physMeas}
Here we describe how the proposed methodology can be used to compute the moments of physical measures. Let Assumption \ref{unique} hold so that there is a unique physical measure $\mu$ with support included in $\Xf$. Therefore for Lebesgue almost every $x \in \Xf$ for which the trajectory of~(\ref{eq:sys}) originating from $x$ stays in $\Xf$ we have for any $f \in C(\Xf)$
\[
\int_{\Xf} f\, d\mu = \lim_{N\to\infty} \frac{1}{N}\sum_{i=1}^N f(T^i(x)).
\]
Selecting $f(x) = x^\alpha$, we can approximately compute the moments of $\mu$ as
\begin{equation}\label{eq:mom_num}
\bs y_\alpha^{\mr{num}} =   \frac{1}{N}\sum_{i=1}^N f(T^i(x)) \approx \int_{\Xf} x^\alpha\, d\mu
\end{equation}
with some $N \gg 1$ and $\alpha$ running over a selected subset of multiindices $\mathcal{I}$. The idea is that the number of moments we (inaccurately) compute using~(\ref{eq:mom_num}) is very small and then we use the optimization problem~(\ref{opt:relax}) to compute a much larger number of moments of $\mu$, just from the information contained in $\{\bs y_\alpha^{\mr{num}} : \alpha \in \mathcal{I}\}$. This is achieved by setting the objective function in~(\ref{opt:relax}) to 
\begin{equation}\label{eq:obj_num}
F(\bs y) = \sum_{\alpha \in \mathcal{I}} (\bs y_\alpha - \bs y_\alpha^{\mr{num}})^2
\end{equation}
or any other metric measuring the discrepancy among moments.

\subsection{Ergodic measures}\label{sec:ergodMeas}
In this section we describe how one can target ergodic measures through the choice of the objective function $F$. In particular, these can be use used to locate embedded unstable fixed points or periodic orbits. The starting point is the well-known fact that ergodic measures are precisely the extreme points of the set of all invariant probability measures (see, e.g., \cite[Proposition 12.4]{phelps2001lectures}; see also~\cite{cross1998approximating} for  general results on the behavior of extreme points under projection in infinite-dimensional vector spaces). Therefore, we can target ergodic measures by selecting a \emph{linear} objective functional in~(\ref{opt:lp_inf}) since the minimum of a linear program is attained at an extreme point, provided the minimizer is unique. This leads to the following immediate result:
\begin{theorem}
Let $F$ be a continuous linear functional such that Assumption \ref{unique} holds and let $\bs y^k \in \Rb^{\binom{n+d_k}{d_k}}$ denote the optimal solution to~(\ref{opt:relax}). Then $\bs y^k$ converges to a moment sequence of an ergodic measure $\mu$, provided Assumption~\ref{as:arch} holds.
\end{theorem}
\begin{proof}
The result follows from Theorem~\ref{thm:conv} and the fact that the minimizer is unique, therefore necessarily an extreme point of the feasible set of~(\ref{opt:lp_inf}).
\end{proof}

A typical choice of the objective functional is $F(\bs y) = \sum_\alpha c_\alpha \bs y_\alpha  $ with $c$ having only finitely many non-zero elements (i.e., $F$ is a linear combination of a finite number of moments).

\subsection{Absolutely continuous measures}
In this section we describe how to target measures $\mu$ absolutely continuous w.r.t. a given measure $\nu$ through the choice of the objective function $F$. In most practical applications the measure $\nu$ will be the Lebesgue measure. The absolutely continuity of $\mu$ w.r.t. $\nu$ is equivalent to the existence of a density $\rho$ such that $\mu = \rho d\nu$. Assuming that $\rho(x) \le \gamma < \infty$ for $\nu$-almost all $x \in X$, we can impose the absolute continuity constraint by choosing
\[
F(\mu) = \mathbb{I}_{\{ \mu \,\mid\, \mu \le \gamma \nu\}},
\]
where $\mathbb{I}_A$ is the extended-value characteristic of a set $A$ (i.e., $\mathbb{I}_A(x) = 0$ if $x \in A$ and $\mathbb{I}_A = +\infty$ if $x\notin A$). Then $F(\bs y)$ is a convex function of $\bs y$ and its domain is the set of moment vectors $\bs y$ such that
\begin{equation}\label{eq:absCont}
M_d(\gamma\bs z - \bs y)\succeq 0 \quad \forall d\in\Nb,
\end{equation}
where $\bs z$ denotes the moment sequence of $\nu$. This is added to the constraints of~(\ref{opt:lp_inf}) and (\ref{opt:relax}) (for a particular fixed $d$ in the latter case).

We remark that, in general, there may be multiple absolutely continuous measures in which case the constraint~(\ref{eq:absCont}) can be combined with an additional choice of the objective functional $F$ in order to target a specific measure (e.g., the physical measure as in Section~\ref{sec:physMeas}).

\subsection{Singular measures}
In this section we describe how to target singular measures through the choice of the objective function $F$. In this work, we focus on atomic measures only. Whether there exists and effective variational characterization of the elusive singular-continuous measures (e.g., the Cantor measure) remains an open problem, to the best of our knowledge.

Atomic measures are of interested because they may correspond to unstable periodic orbits or fixed points, which are difficult to obtain using simulation-based techniques. These structures are typically ergodic and therefore the methods of this section can be combined with those of Section~\ref{sec:ergodMeas}.

We use the observation that if a measure $\mu$ consists of $K$ atoms, then the associated moment matrix $M_d(\bs y)$ is of rank at most $K$ for all $d$ and of rank exactly $K$ for sufficiently large $d$. Therefore, in order to seek an invariant measure consisting of $K$ atoms (e.g., a $K$-period orbit) one would in principle want to choose
\[
F(\bs y) = \mathbb{I}_{\{ \bs y \,\mid\, \mr{rank}\,M_d(\bs y) = K\}}.
\]
Unfortunately, such $F$ is not convex and therefore we propose to use a convex relaxation
\[
F(\bs y) = \mathbb{I}_{\{ \bs y \,\mid\, \mr{trace}\,M_d(\bs y) \le \gamma\}},
\]
where $\gamma \ge 0$ is a regularization parameter.  This translates to the constraint
\[
\mr{trace}\,M_d(\bs y) \le \gamma
\]
which is added to the constraints of~(\ref{opt:lp_inf}) and (\ref{opt:relax}).
We note that  since $M_d(\bs y)$ is positive semidefinite, its trace coincides with its nuclear norm, which is a standard proxy for rank minimization \cite{fazel2002} (since the convex hull of the rank is the trace on the unit ball of symmetric matrices).
 
\section{Continuous time version}\label{sec:contDet}
In this section we briefly outline how the presented approach extends to continuous time. Assume therefore that we are dealing with the dynamical system of the form
\begin{equation}\label{eq:sys_cont}
\dot{x} = b(x),
\end{equation}
each component of the vector field $b$ is assumed to be a multivariate polynomial. We are seeking a non-negative measure $\mu$ on $\bs X$ invariant under the flow of dynamical system~(\ref{eq:sys_cont}), which is equivalent to the condition
\begin{equation}\label{eq:invarCont}
\int_{\Xf} \mr{grad} f \cdot b\, d\mu = 0
\end{equation}
for all $f \in C^1(\Xf)$.

The infinite-dimensional convex optimization problem~(\ref{opt:lp_inf}) then becomes
\begin{equation}\label{opt:lp_inf_cont}
\begin{array}{lll}
 \min\limits_{\mu \in \mathcal{M}_+(\mathbf{X})} & F(\mu) \\
 \mathrm{s.t.} & \int_{\Xf} \mr{grad} f \cdot b\, d\mu = 0 \quad \forall f\in C(\mathbf{X}) \vspace{1mm}\\
&  \int_\mathbf{X} d\mu = 1.\\
\end{array}
\end{equation}
This optimization problem is then approximated by taking $f = x^\alpha$, $\alpha \in\mathbb{N}^n$, and proceeding in exactly the same way as described in Section~\ref{sec:approx}, leading to a finite-dimensional relaxation of the same form as~(\ref{opt:relax}), with the same convergence results of Theorems~\ref{thm:conv} and \ref{thm:weakDense}.

\section{Markov processes}\label{sec:markov}
In this section we describe a generalization to Markov processes evolving on the state-space $\bf X$. We assume a Markov chain in the state-space form\footnote{For a relation of this form of a Markov process to the one specified by the transition kernel, see, e.g.,~\cite{HernandezLasserre}.}
\[
x_{k+1} = T(x_k,w_k),
\]
where $(w_k)_{k=0}^\infty$ is a sequence of independent identically distributed random variables
with values in a given set $\bs W$ and the mapping $T$ is assumed to be a polynomial in $(x,w)$. The distribution of the random variables $w_k$ is denoted by $P_w$, i.e., for all Borel $\bf A\subset \bf X$, $P_w(\bf A)$ is the probability that $w_k \in \bf A$.

The condition for a probability measure $\mu$ to be invariant then reads
\begin{equation}\label{eq:invar_markov}
\int_{\bf X}\int_{\bs W} f (T(x,w))\, dP_w(w)  \,d\mu(x) = \int_{\mathbf{X}} f(x) \,d\mu(x)
\end{equation}
for all $f \in C(\mathbf{X})$. This equation is \emph{linear} in $\mu$ and can therefore be used in~(\ref{opt:lp_inf}) instead of the first equality constraint. The approach then proceeds along the steps of Section~\ref{sec:approx}, i.e., we set $f(x) = x^\alpha$ and enforce (\ref{eq:invar_markov}) for all such $f$ with $|\alpha|\le k$, leading to
 \begin{equation}\label{eq:markTrunc}
 \int_{\bs X} \int_{\bs W} T^\alpha(x,w)\,dP_w(w)\,d\mu(x)  =  \int_{\bs X}x^\alpha \,d\mu ,\quad |\alpha | \le k.
  \end{equation}
Since $T$ is a polynomial in $(x,w)$, $T^\alpha$ can be written as
\[
T^\alpha(x,w) = \sum_{\beta,\gamma} t_{\alpha,\beta,\gamma}\, x^\beta w^\gamma
\]
for some coefficients $t_{\alpha,\beta,\gamma}$. Therefore,
\[
\int_{\bs W} T^\alpha(x,w)dP_w(w) = \sum_{\beta,\gamma} t_{\alpha,\beta,\gamma}\, x^\beta \int_{\bs W} w^\gamma\,dP_w(w) = \sum_{\beta,\gamma} t_{\alpha,\beta,\gamma}\, x^\beta  m_{\gamma},
\]
where
\[
m_{\gamma} = \int_{\bs W} w^\gamma\,dP_w(w) 
\]
are the moments of $w_k$, which are fixed numbers that can be either precomputed analytically or using sampling  techniques. The equation~(\ref{eq:markTrunc}) can therefore be re-written as
\[
\sum_{\beta,\gamma} t_{\alpha,\beta,\gamma} m_\gamma  \int_{\bs X} x^\beta d\mu(x)  = \int_{\bs X} x^\alpha\,d\mu(x),d\mu ,\quad |\alpha | \le k,
\]
or, in terms of the moments of $\mu$~(\ref{eq:mom_def}),
\[
\sum_{\beta,\gamma} t_{\alpha,\beta,\gamma} m_\gamma  \bs y_\beta  = \bs y_\alpha, \quad |\alpha | \le k.
\]
This is a finite-dimensional system of linear equations of the form
\[
A\bs y = 0.
\]
Adding the normalization constraint $\bs y_0 = 1$ and the positive-semidefiniteness constraints $M_{d_k}(\bs y)\succeq 0$ and $M_{d_k}(\bs y,g_i)\succeq 0$ leads to the optimization problem~(\ref{opt:relax}). The objective functional $F(\bs y)$ of~(\ref{opt:relax}) is again chosen in order to target a particular class of invariant measures. The same convergence guarantees of Theorems~\ref{thm:conv} and \ref{thm:weakDense} hold.

\begin{remark}[Uniqueness]
It is interesting to note that, in the presence of randomness, it is much more common for a \emph{unique} invariant measure to exist. For example, a sufficient condition for this is the recurrence of the Markov chain; see~\cite[Chapter 10]{meyn2012markov} for more details.
\end{remark}

\subsection{Continuous-time Markov processes}\label{sec:contStoch}
The extension to continuous-time stochastic processes is straightforward. The invariance condition~(\ref{eq:invarCont}) is simply replaced by
\begin{equation}\label{eq:invMarkCont}
\int_{\bf X} \mathcal{A}f\,d\mu = 0,
\end{equation}
where $\mathcal{A}$ is the infinite-dimensional generator of the process. For concreteness, let us consider the stochastic differential equation
\[
dX_t = b(X_t)dt + \sigma(X_t) dW_t,
\]
where $b:\Rb^n\to\Rb^n$ is the drift, $\sigma :\Rb^n \to \Rb^{n\times k}$ the diffusion matrix
 and $W_t$ is a vector-valued Wiener process. Then we have
\[
\mathcal{A}f = \sum_i b_i \frac{\partial f}{\partial x_i} + \frac{1}{2}\sum_{i,j}[\sigma\sigma^\top]_{i,j} \frac{\partial^2 f}{\partial x_i\partial x_j},
\]
which is a polynomial provided that $b(x)$ and $\sigma(x)$ are polynomial in $x$ and we set $f(x) = x^\alpha$. The invariance condition~(\ref{eq:invMarkCont}) can therefore be expressed solely in terms of the moments of $\mu$~(\ref{eq:mom_def}). The approach then proceeds in exactly the same fashion as described in Section~\ref{sec:approx}, leading to a finite-dimensional relaxation of the same form as~(\ref{opt:relax}).
\new{\paragraph{Certifying non-existence of an invariant measure} Contrary to the deterministic case, stochastic processes driven by a Wiener process typically evolve on non-compact domains, thereby rendering the question of the \emph{existence} of an invariant measure much more subtle. For example, the Wiener process itself does not admit an invariant measure but the Ornstein-Uhlenbeck process $dx(t) = -ax(t)dt + \sigma dW(t)$ does for any $a > 0$. Interestingly, the proposed approach provides a means to numerically certify that no invariant measure exists. Indeed, since the feasible set of the SDP relaxation~(\ref{opt:relax}) contains the truncated moment sequences of all invariant measures, proving the emptiness of this feasible set (which is a finite-dimensional spectrahedron) implies the \emph{non-existence} of an invariant measure. Indication of such infeasibility is detected during the solution of SDP~(\ref{opt:relax}) by most existing solvers, although its rigorous certification is more involved~\cite{LMIexact}. The proposed method thereby complements methods based on sum-of-squares programming that can be used to prove the existence of an invariant measure using Foster-Lyapunov conditions~\cite{meyn2012markov}.}

%

\section{Eigenmeasures of Perron-Frobenius}\label{sec:pf}
In this section we briefly describe how the presented approach can be extended to computation of the eigenmeasures of the Perron-Frobenius operator~\cite{mezic2004comparison}. For concreteness we work with the discrete-time dynamics~(\ref{eq:sys}), although analogous results can be obtained for continuous time, following the developments of Section~\ref{sec:contDet}. The Perron-Frobenius operator $\mathcal{P} : \bs M_c(\bs X)\to \bs M_c(\bs X)$ is defined by
\[
(\mathcal{P}\mu)(A) = \mu(T^{-1}(A))
\]
for every Borel measurable set $A\subset \bs X$. Here $M_c(\bs X)$ stands for the vector space of all  complex-valued measures on $\bs X$. Given $\lambda \in \mathbb{C}$, a complex-valued measure $\mu$ is an eigenmeasure of $\mathcal{P}$ if 
\begin{equation}\label{eq:eigmeas}
\int f\circ T\,d\mu = \lambda \int f\,d\mu
\end{equation}
for all $f\in C(\bs X)$. Since the moment-based approach developed previously applies to nonnegative measures, we use the Jordan decomposition
\[
\mu = \mu_R^+ - \mu_R^{-} +i(\mu_I^+ - \mu_I^{-}),
\]
where $i = \sqrt{-1}$ is the imaginary unit and $\mu_R^+ \in \bs M(\bs X)$, $\mu_R^- \in \bs M(\bs X)$, $\mu_I^+ \in \bs M(\bs X)$, $\mu_I^- \in \bs M(\bs X)$. Similarly, we write
\[
\lambda = \lambda_R + i\lambda_I.
\]
Then the condition~(\ref{eq:eigmeas}) is equivalent to
\begin{subequations}
\begin{align}
&\int (f\circ T - \lambda_R f)\,d\mu_R^+ -\int (f\circ T - \lambda_R f)\,d\mu_R^-  + \lambda_I\int f\,d\mu_I^- -  \lambda_I\int f\,d\mu_I^+ = 0, \\
&\int (f\circ T - \lambda_R f)\,d\mu_I^+ -\int (f\circ T - \lambda_R f)\,d\mu_I^-  + \lambda_I\int f\,d\mu_R^- -  \lambda_I\int f\,d\mu_R^+ = 0
\end{align}
\end{subequations}
for all $f\in C(\bs X)$. When expressed with $f = x^\alpha$, $\alpha \in \mathbb{N}^n$, this is equivalent to 
\[
A(\bs y_R^+,\bs y_R^-,\bs y_I^+,\bs y_I^-) = 0,
\]
where $\bs y_R^+$ etc are the moment sequences of the respective measures and $A$ is a linear operator. Coupled with the normalization constraint
\[
(\bs y_R^+)_0 + (\bs y_R^-)_0 + (\bs y_I^+)_0 + (\bs y_I^-)_0 = 1
\]
and and objective functional $F(\bs y_R^+,\bs y_R^-,\bs y_I^+,\bs y_I^-)$, we arrive at an infinite dimensional linear programming problem
\begin{equation}\label{opt:lp_inf_PF}
\begin{array}{lll}
 \min\limits_{\bs y_R^+,\bs y_R^-,\bs y_I^+,\bs y_I^-} & F(\bs y_R^+,\bs y_R^-,\bs y_I^+,\bs y_I^-) \\
 \mathrm{s.t.} & A(\bs y_R^+,\bs y_R^-,\bs y_I^+,\bs y_I^-) = 0\\
&  (\bs y_R^+)_0 + (\bs y_R^-)_0 + (\bs y_I^+)_0 + (\bs y_I^-)_0 = 1\\
& y_R^+ \in \bs M(\bs X),\;  y_R^- \in \bs M(\bs X),\;  y_I^+ \in \bs M(\bs X),\; y_I^- \in \bs M(\bs X),
\end{array}
\end{equation}
which is then approximated by a sequence of finite-dimensional SDPs in exactly the same fashion as described in Section~\ref{sec:approx}, with the convergence results of Theorem~\ref{thm:conv} also holding in this setting.

\section{Numerical examples}\label{sec:numEx}
\subsection{Logistic map}
As our first example we consider the Logistic map:
\[
x^+ = 2x^2 - 1
\]
on the set $\Xf = [-1,1] = \{x \in \Rb : (x+1)(1-x)\ge 0 \}$. 

\subsubsection{Physical measure}
First, we compute the moments of the unique physical measure $\mu$ on $\Xf$. For this example, the density of the physical measure is given by \cite{LasotaMackey1994}
\[
\rho(x) = \frac{1}{\pi}\frac{1}{\sqrt{1-x^2}}.
\]
We used the first moment $\frac{1}{\pi}\int_{-1}^1 \frac{x}{\sqrt{1-x^2}}\, dx = 0$ as data input for the objective function of the form (\ref{eq:obj_num}), i.e., we set $F(\bs y) = (\bs y_1 - 0)^2$ (note that  instead of an exactly computed value we could have used an imprecise value of the moment from a simulation or determine this value based on symmetry without analytically integrating the density). Then we solve~(\ref{opt:relax}) with $k \in\{5,10,100\}$ and compute a degree $k$ polynomial approximation to the density using~(\ref{eq:rhobar}); as in~\cite{henrion2012kybernetika}, for numerical stability reasons, we work in the Chebyshev basis rather than the monomial basis (i.e., we express the constraint~(\ref{eq:eq_const}) using the Chebyshev basis polynomials instead of monomials $x^\alpha$ and replace the monomial basis vectors $v_{d_i}$ by the vectors of Chebyshev polynomials up to degree $d_i$ in~(\ref{eq:locMat})); see~\cite{henrion2012kybernetika} for more details on the use of Chebyshev polynomials in this context. In Figure~\ref{fig:logistic} we compare the true density and the polynomial approximations. We observe a very good fit even for low-degree approximations and oscillations of the sign of the approximation error, akin to classical results from approximation theory. The computed moments (transformed to the monomial basis) are compared in Table~\ref{tab:log}; we see a very good match.

\begin{figure*}[th]
\begin{picture}(140,130)
\put(-10,0){\includegraphics[width=60mm]{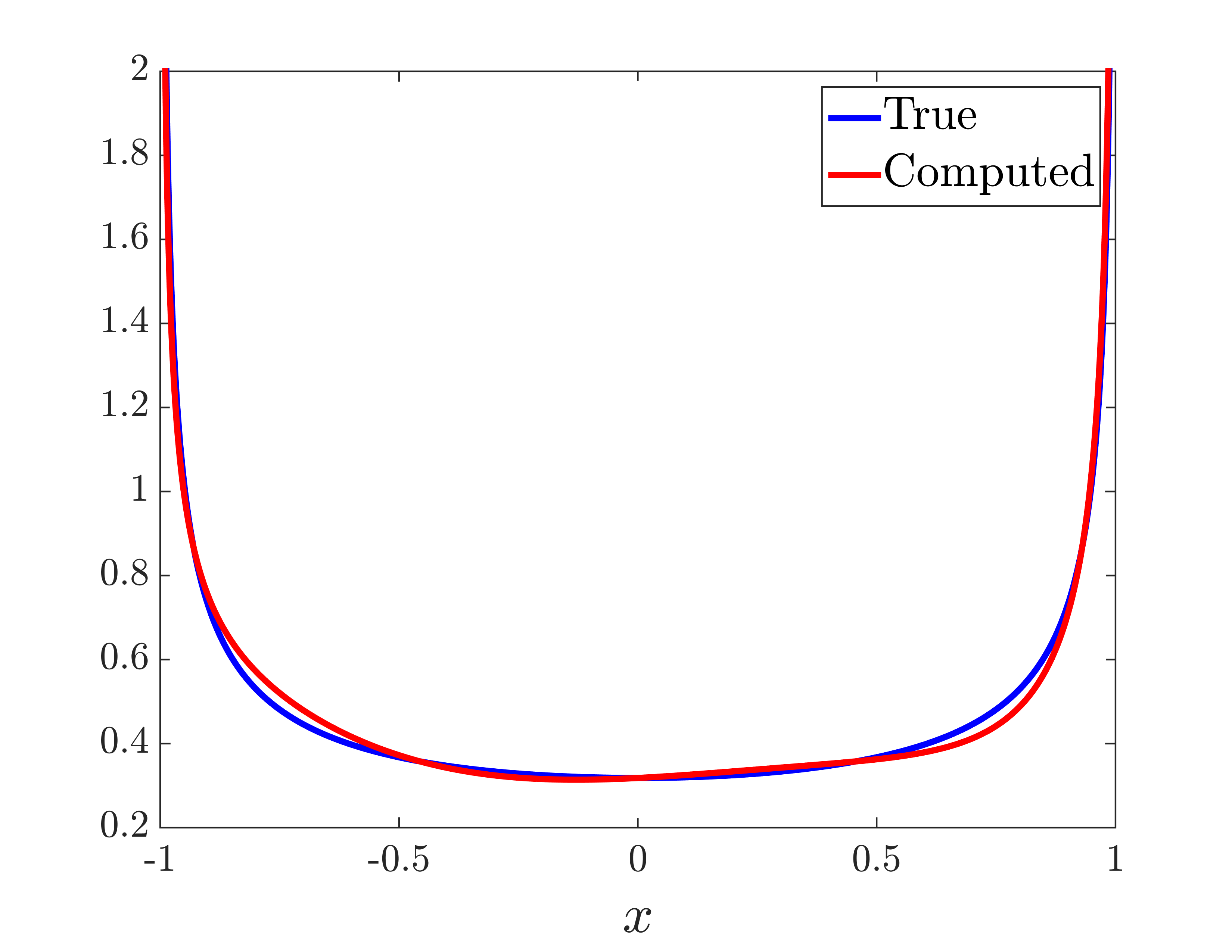}}
\put(150,0){\includegraphics[width=60mm]{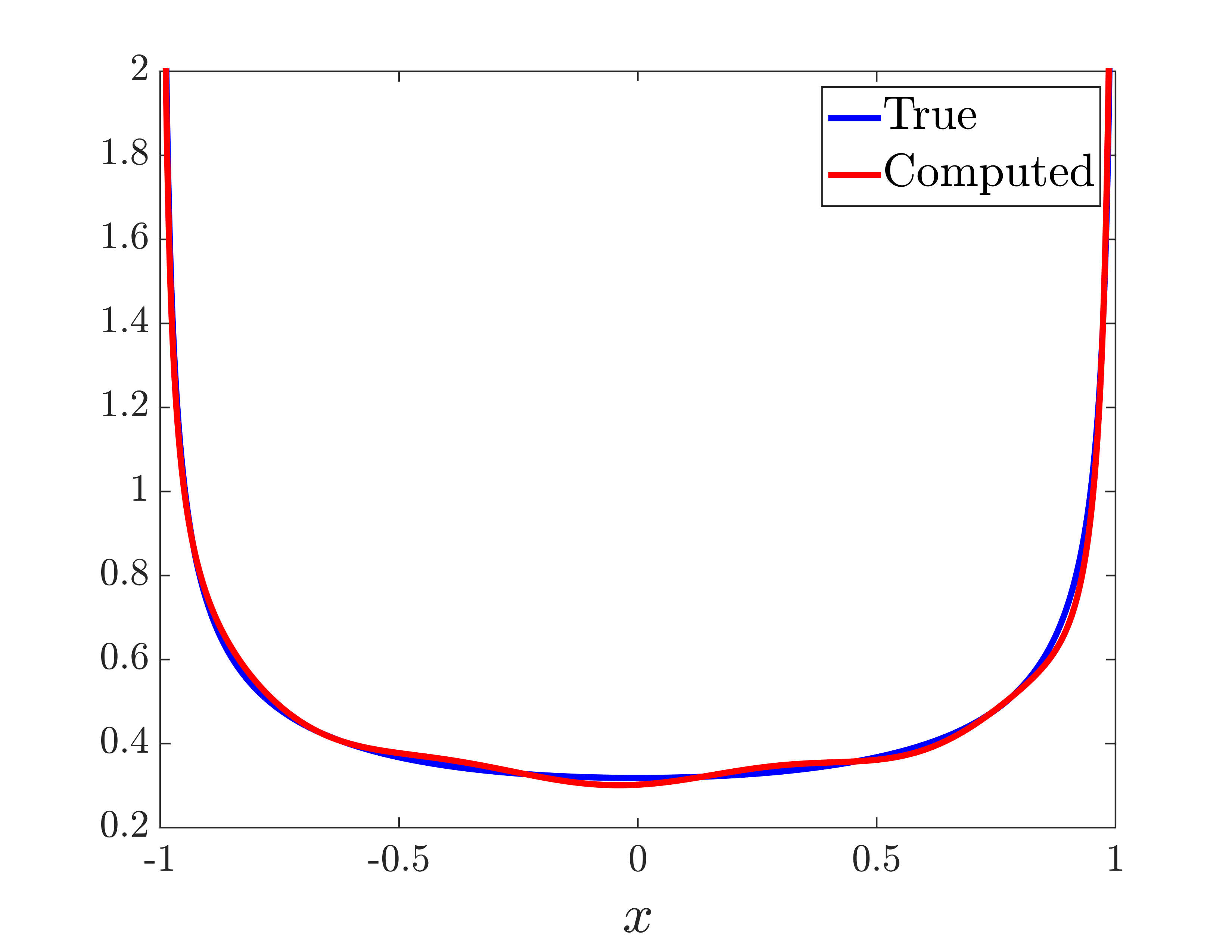}}
\put(310,0){\includegraphics[width=60mm]{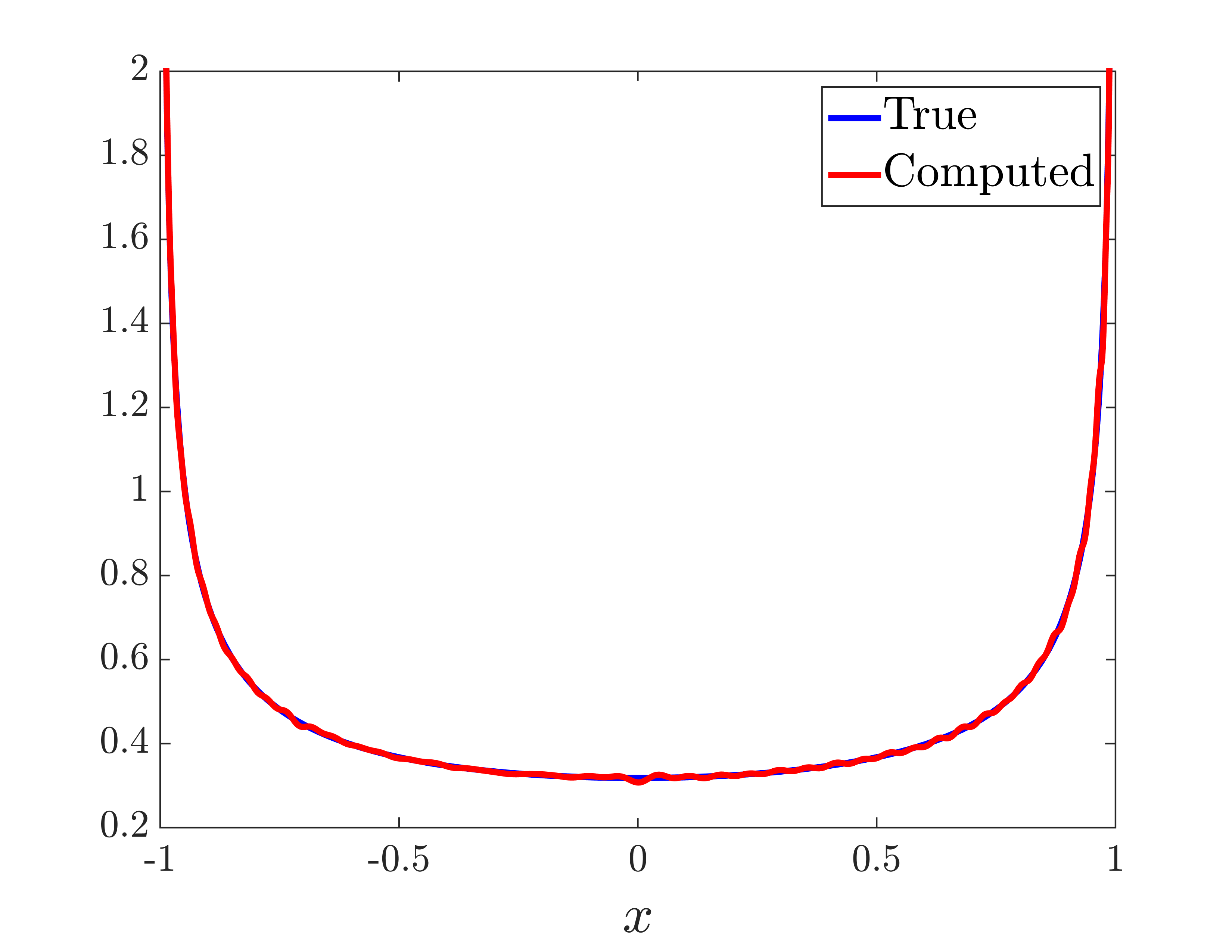}}

\put(65,85){\footnotesize $k = 5$}
\put(225,85){\footnotesize $k = 10$}
\put(385,85){\footnotesize $k = 100$}

\end{picture}
\caption{\small Logistic map: approximation of the density. Only the first moment was given as input to optimization problem~(\ref{opt:relax}).}
\label{fig:logistic}
\end{figure*}

\begin{table}[ht]
\centering
\caption{\rm \small Logistic map: comparison of moments computed by the SDP relaxation~(\ref{opt:relax}) with the true moments computed analytically. Only the first moment was given as data input to~(\ref{opt:relax}).}  \label{tab:log}\vspace{2mm}
\begin{tabular}{ccccccccccc}
\toprule
Moments & $x$ & $x^2$ & $x^3$ & $x^4$ & $x^5$ & $x^6$ & $x^7$ &  $x^8$ & $x^9$ & $x^{10}$ \\\midrule
SDP & 0.0000 &   0.5000 &    0.0047 &    0.3750   & 0.0062  &  0.3131 &   0.0068 & 0.2746 &   0.0071 &   0.2477 \\\midrule
True & 0.0000  &  0.5000  &   0.0000  &   0.3750 &  0.0000 &    0.3125 &   0.0000 & 0.2734  &       0.0000 &  0.2461 \\
\bottomrule
\end{tabular}
\end{table}


\subsection{H\' enon map}
Our second example is the H\' enon map:
\begin{align*}
x_1^+ &=  1 - 1.4x_1^2 + x_2\\
x_2^+ &= 0.3x_1.
\end{align*}
The set $\Xf$ is the box $[-1.5,1.5]\times [-0.4,0.4]$ expressed as
\[
\Xf = \{ x\mid (x_1-1.5)(1.5-x_1) \ge 0,\; (x_2-0.4)(0.4-x_2) \ge 0\}.
\]
The goal is to compute the moments of the physical measure $
\mu$ on $\Xf$. The available information is only the first moment $\bs y^{\mr{num}}_{(1,0)}$ (i.e., the expectation of the first coordinate) of $\mu$ approximately computed using~(\ref{eq:mom_num}) with $f(x) = x_1$. The objective function is $F(\bs y) = (\bs y_{(1,0)} - \bs y^{\mr{num}}_{(1,0)})^2$. Table~\ref{tab:henon} compares the moments returned by relaxation~(\ref{opt:relax}) of order $k = 10$ with moments computed numerically using~(\ref{eq:mom_num}); we observe a good agreement. Figure~\ref{fig:henon} then compares the support approximations computed using using the Christoffel polynomial~(\ref{eq:confReg}) with $\bs y = \bs y^k$.

\begin{table}[ht]
\centering
\caption{\rm \small H\' enon map: comparison of moments computed by the SDP relaxation~(\ref{opt:relax}) and using~(\ref{eq:mom_num}). Only the first moment corresponding to $x_1$ was given as data input to~(\ref{opt:relax}).}  \label{tab:henon}\vspace{2mm}
\begin{tabular}{cccccccccc}
\toprule
Moments & $x_1$ & $x_2$ & $x_1^2$ & $x_1x_2$ & $x_2^2$ & $x_1^3$ & $x_1^2x_2$ &  $x_1x_2^2$ & $x_2^3$ \\\midrule
SDP& 0.2570 & 0.0771 & 0.5858 & -0.0379 & 0.0527 & 0.2468 & 0.0131 & -0.0140 & 0.0067 \\\midrule
Numeric & $0.2570
$ & $0.0771$ & $0.5858$ & -0.0291 & $0.0527$ & $0.2320$ & $0.0510$ &-0.0174 & $0.0063$ \\
\bottomrule
\end{tabular}
\end{table}

\new{
\subsection{Lorenz system}
Next, we consider the classical Lorenz system
\begin{align*}
	\dot{x}_1 &= 10(x_2-x_1) \\
	 \dot{x}_2 & = x_1(28-x_3)-x_2 \\
	  \dot{x}_3 &= x_1x_2-\frac{8}{3}x_3
\end{align*}
scaled by the linear coordinate transformation $\hat x = \mr{diag}([1/25,1/30,1/50]) x$. The goal is to compute the moments of the physical measure $\mu$ which is supported on the Lorenz attractor. Here we investigate the effect of the amount of information available for the computation in terms of the number of moments used in the objective function $F(\bs y) = \sum_{i=1}^N  (\bs y_i - \bs y^{\mr{num}}_i)^2  $, where $\bs y^{\mr{num}}_i$ are approximations of the true moments computed using simulation. We investigate as well the effect of the accuracy of the approximate moments provided by varying the simulation length $M$. The accuracy is measured in terms of the percentage root mean square error on the sequence of the first 56 moments (i.e., moments up to degree 5) with respect the numerical approximations from a simulation of length $10^6$. Figure~\ref{fig:lorenz} shows the results for $M = 10^2$ and $M=10^6$ with $N \in \{1,\ldots,10\}$. We observe a rather small impact of the accuracy of the moments provided (determined by $M$); on the other hand, the number of moments provided $N$ plays a significant role, resulting in an error of roughly $15 \%$ with one moment provided and error of roughly $0.05 \%$ with seven or more moments.

\begin{figure*}[h!]
\begin{picture}(140,180)
\put(130,3){\includegraphics[width=80mm]{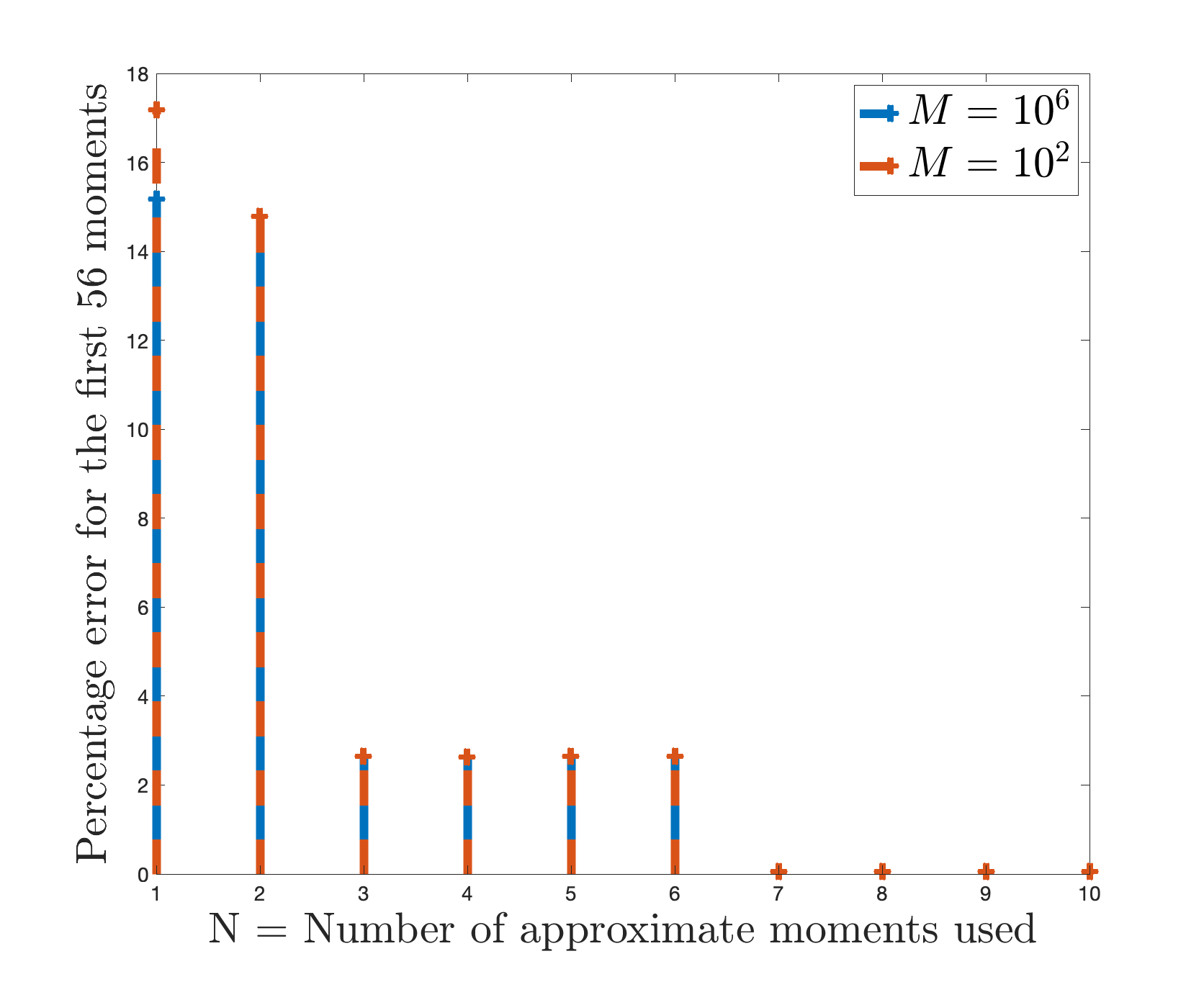}} 
\end{picture}
\caption{Approximation of the support of the invariant measure of the H\'enon map, confidence regions of the Christoffel polynomial built from 20 moments : green 90$\,\%$, red 50$\,\%$, blue simulated trajectories.}
\label{fig:lorenz}
\end{figure*}

}

\begin{figure*}[t!]
\begin{picture}(140,200)
\put(130,3){\includegraphics[width=90mm]{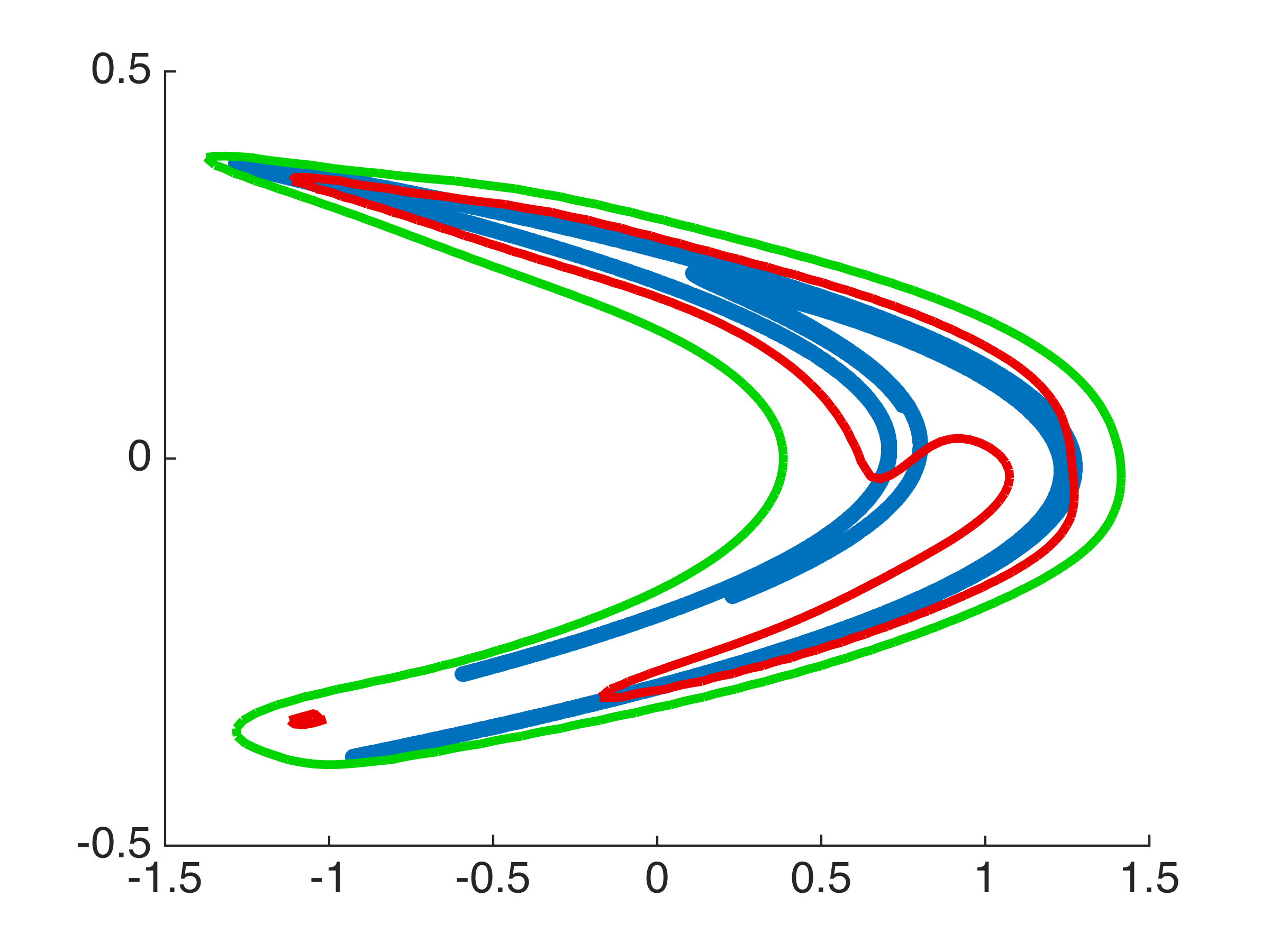}} 
\put(258,-2){ $x_1$}
\put(135,98){$x_2$}
\end{picture}
\caption{Approximation of the support of the invariant measure of the H\'enon map, confidence regions of the Christoffel polynomial built from 20 moments : green 90$\,\%$, red 50$\,\%$, blue simulated trajectories.}
\label{fig:henon}
\end{figure*}

\new{
\subsection{Stochastic processes}
\paragraph{Ornstein–Uhlenbeck} To demonstrate the approach of Section~\ref{sec:contStoch}, we first choose the classical Ornstein–Uhlenbeck process given by
\[
dx(t) = -ax(t)dt + \sigma dW(t),
\]
where $W(t)$ is the Wiener process. For $a > 0$, this process has a unique invariant measure equal to $\mathcal{N}(0,\frac{\sigma^2}{2\theta})$, i.e., the normal distribution with zero mean and standard deviation $\frac{\sigma}{\sqrt{2\theta}}$. In order to compute the moments, we solved~(\ref{opt:relax}) without the objective function (i.e., as a feasiblity problem). The results for $a =1$ and $\sigma =1$ and $d = 10$ are shown in Table~\ref{tab:Ornstein}; we observe a precise match.

\begin{table}[ht]
\centering\small
\caption{\rm \small \new{Ornstein–Uhlenbeck process: comparison of moments computed by the SDP relaxation~(\ref{opt:relax}) with moments computed analytically.}} \label{tab:Ornstein}\vspace{2mm}
\begin{tabular}{ccccccccccc}
\toprule
Moments & $x$ & $x^2$ & $x^3$ & $x^4$ & $x^5$ & $x^6$ & $x^7$ &  $x^8$ & $x^9$ & $x^{10}$ \\\midrule
SDP &  0 & 0.5 & 0 & 0.75 &  0 & 1.8750 & 0 & 6.5625 & 0 & 29.5312 \\\midrule
True & 0 & 0.5 & 0 & 0.75 &  0 & 1.8750 & 0 & 6.5625 & 0 & 29.5312\\
\bottomrule
\end{tabular}
\end{table}

\paragraph{Nonlinear drift} Next, we test the approach on a stochastic differential equation with a nonlinear drift
\[
dx(t) = -ax^3(t) dt + \sigma dW(t).
\]
For $a > 0$, this stochastic process is uniformly ergodic and hence admits a unique invariant measure~\cite{meyn2012markov}. Table~\ref{tab:nonlinearDrift} shows the results of solving~(\ref{opt:relax}) for $d = 10$ in comparison with Monte Carlo simulation using the Euler forward discretization and $10^5$ sample trajectories.

\begin{table}[ht]
\centering\small
\caption{\rm \small \new{Nonlinear drift: comparison of moments computed by the SDP relaxation~(\ref{opt:relax}) with moments computed using a Monte Carlo simulation.}} \label{tab:nonlinearDrift}\vspace{2mm}
\begin{tabular}{ccccccccccc}
\toprule
Moments & $x$ & $x^2$ & $x^3$ & $x^4$ & $x^5$ & $x^6$ & $x^7$ &  $x^8$ & $x^9$ & $x^{10}$ \\\midrule
SDP &  -0.0000 &   0.4758 &    0.0000 &   0.5000 &  -0.0000 &    0.7138 &   0.0000 &    1.2500 &  -0.0000 &   2.4982 \\\midrule
Mont Carlo & 0.0001  &  0.4776  &  0.0011  &  0.4997  &  0.0039  &  0.7192 &   0.0130  &  1.2637  &  0.0413  &  2.5683\\
\bottomrule
\end{tabular}
\end{table}

\paragraph{Detecting nonexistence of invariant measures}
In this section we demonstrate the ability of the approach to numerically prove the non-existence of an invariant measure as described in Section~\ref{sec:contStoch}. For this, we consider the process
\[
dx(t) = -ax^3(t) dt + (1+x^2(t))dW(t).
\]
Clearly, the process does not admit an invariant measure for $a < 0$. The situation is more interesting for $a >0$. We carried out computations for $a \in \{0,1,\ldots,10\}$, solving the SDP relaxation~(\ref{opt:relax}). The results indicated that no invariant measure exists for this range of values of $a$, although the degree to certify it (i.e., to render~(\ref{opt:relax}) infeasible) increases with $a$. This is intuitive since, vaguely speaking, the larger the value of $a$, the ``more stable'' the process is. Table~\ref{tab:nonlinearDrift} shows this dependence; it appears that the minimum degree obeys the linear relation $d = 2a+2$, although we did not attempt to prove this analytically. It should be noted that the infeasibility of~(\ref{opt:relax}) was decided by the interior point solver SeDuMi; in particular we did not use rigorous certification tools such as \cite{LMIexact} and therefore these results should be understood as a strong numerical evidence of non-existence rather than a rigorous proof.

\begin{table}[ht]
\centering\small
\caption{\rm \small \new{Detecting nonexistence of invariant measure: minimum degree required to render the SDP relaxation~(\ref{opt:relax}) infeasible versus the value of the drift coefficient $a$.}} \label{tab:nonlinearDrift}\vspace{2mm}
\begin{tabular}{cccccccccccc}
\toprule
drift coefficient $a$ &  0 &   1 &    2 &   3 &  4 &    5 &   6 &    7 & 8 &   9 & 10  \\\midrule
minimum degree for infeasiblity of~(\ref{opt:relax})  & 2  &  4  &  6  &  8  &  10  &  12 &   14  &  16  &  18  &  20 & 22 \\
\bottomrule
\end{tabular}
\end{table}

}

\new{\section{Conclusion}
This work presented a convex-optimization-based method for computation of invariant measures for continuous and discrete time deterministic and stochastic systems. We described how to cast the problem of invariant set computation as an infinite-dimensional LP in the space of Borel measures and how to target particular invariant measures by the choice of an objective functional. We showed how this infinite-dimensional LP can be approximated by a sequence of finite-dimensional SDPs with a guaranteed asymptotic convergence and how the results of this SDP can be used for support approximation of the invariant measure using the Christoffel-Darboux kernel.  Interesting by-products of the approach are a method to certify non-existence of invariant measures and a method to compute eigenmeasures of the Perron-Frobenius operator.

Future work should focus on improving the scalability of the approach by exploit sparsity or symmetries of the problem at hand or on developing a data-driven counterpart of the approach where the model is unknown and only finite collection of observations is available, in the spirit of~\cite{kodaMCIdata}.

}

\section*{Acknowledgment}

This work benefited from discussions with Victor Magron. This research was supported in part by the ARO-MURI grant W911NF-17-1-0306. The research of M. Korda was also supported by the Swiss National Science Foundation under grant P2ELP2\_165166.


\begin{thebibliography}{10}

\bibitem{bochi2018}
J. Bochi.
\newblock{Ergodic optimization of Birkhoff averages and Lyapunov exponents.}
\newblock{Proceedings of the International Congress of Mathematicians, 2018.}


\bibitem{bollt2005path}
E. M. Bollt.
\newblock The path towards a longer life: On invariant sets and the escape time
  landscape.
\newblock {International Journal of Bifurcation and Chaos},
  15(05):1615--1624, 2005.

\bibitem{chernyshenko} 
S. I. Chernyshenko, P. Goulart, D. Huang, A. Papachristodoulou.  \newblock{Polynomial sum of squares in fluid dynamics: a review with a look ahead.} 
\newblock{Philosophical Transactions of the Royal Society A: Mathematical, Physical and Engineering Sciences 372.2020 (2014): 20130350.}


\bibitem{cross1998approximating}
W. P. Cross, H. E. Romeijn, R. L. Smith.
\newblock{Approximating extreme points of infinite dimensional convex sets.}
\newblock{Mathematics of Operations Research, 23(2):433--442, 1998.}

\bibitem{fantuzzi}
G. Fantuzzi, D. Goluskin,  D. Huang, S. I. Chernyshenko. 
\newblock{Bounds for deterministic and stochastic dynamical systems using sum-of-squares optimization.}
 \newblock{ {SIAM} Journal on Applied Dynamical Systems, 15(4), 1962-1988.}



\bibitem{fazel2002}
M. Fazel. Matrix rank minimization with applications. PhD Thesis, Elec. Eng. Dept, Stanford University, 2002. 


\bibitem{gaitsgory2009linear}
V. Gaitsgory, M. Quincampoix.
\newblock Linear programming approach to deterministic infinite horizon optimal
  control problems with discounting.
\newblock {SIAM Journal on Control and Optimization}, 48(4):2480--2512,
  2009.

\bibitem{goluskin}
D. Goluskin. 
\newblock{Bounding averages rigorously using semidefinite programming: mean moments of the Lorenz system.}
\newblock{ Journal of Nonlinear Science, 28(2), 621-651.}

\bibitem{gustafsson2009bergman}
B. Gustafsson, M. Putinar, E. B. Saff, N. Stylianopoulos.
\newblock Bergman polynomials on an archipelago: estimates, zeros and shape
  reconstruction.
\newblock {Advances in Mathematics}, 222(4):1405--1460, 2009.

\bibitem{henrion2012kybernetika}
D. Henrion.
\newblock Semidefinite characterisation of invariant measures for
  one-dimensional discrete dynamical systems.
\newblock {Kybernetika}, 48(6):1089--1099, 2012.


\bibitem{gloptipoly3}
D. Henrion, J. B. Lasserre, J. L{\"o}fberg.
\newblock Gloptipoly 3: moments, optimization and semidefinite programming.
\newblock {Optimization Methods and Software}, 24:761--779, 2009.

\bibitem{HernandezLasserre}
O. Hern\'andez-Lerma, J. B. Lasserre.
\newblock {Discrete-time Markov control processes: basic optimality
  criteria}.
\newblock Springer, 1996.

\bibitem{jenkinson2006ergodic}
O. Jenkinson.
\newblock Ergodic optimization.
\newblock {Discrete and Continuous Dynamical Systems}, 15(1):197, 2006.

\bibitem{jenkinson2019}
O. Jenkinson.
\newblock{Ergodic optimization in dynamical systems.}
\newblock{Ergodic Theory and Dynamical Systems 39.10 (2019): 2593-2618.}



\bibitem{junge2017sighting}
O. Junge, I. G. Kevrekidis.
\newblock On the sighting of unicorns: A variational approach to computing
  invariant sets in dynamical systems.
\newblock {Chaos: An Interdisciplinary Journal of Nonlinear Science},
  27(6):063102, 2017.

\bibitem{kodaMCIdata}
Korda, Milan. 
\newblock{Computing controlled invariant sets from data using convex optimization.}
\newblock{SIAM Journal on Control and Optimization,  arXiv preprint arXiv:1912.03256 (2019).}


\bibitem{kordaMCI}
M. Korda, D. Henrion, C. N. Jones.
\newblock Convex computation of the maximum controlled invariant set for
  polynomial control systems.
\newblock {SIAM Journal on Control and Optimization}, 52(5):2944--2969,
  2014.

\bibitem{korda2016controller}
M. Korda, D. Henrion, C. N. Jones.
\newblock Controller design and value function approximation for nonlinear
  dynamical systems.
\newblock {Automatica}, 67:54--66, 2016.

\bibitem{korda2017convergence_opt}
M. Korda, D. Henrion, C. N. Jones.
\newblock Convergence rates of moment-sum-of-squares hierarchies for optimal
  control problems.
\newblock {\em Systems \& Control Letters}, 100:1--5, 2017.

\bibitem{LasotaMackey1994}
A. Lasota, M. C. Mackey.
Chaos, fractals, and noise -
Stochastic aspects of dynamics.
Springer, 1994.

\bibitem{LMIexact}
Henrion, Didier, Simone Naldi, and Mohab Safey El Din. 
\newblock Exact algorithms for linear matrix inequalities.
\newblock {{SIAM} Journal on Optimization}, 26.4 (2016): 2512-2539.


\bibitem{lasserreBook}
J. B. Lasserre.
\newblock {Moments, positive polynomials and their applications,}.
\newblock Imperial College Press, 2010.

\bibitem{lasserre2001global}
J. B. Lasserre.
\newblock Global optimization with polynomials and the problem of moments.
\newblock {SIAM Journal on Optimization}, 11(3):796--817, 2001.

\bibitem{lasserre2008nonlinear}
J. B. Lasserre, D. Henrion, C. Prieur, E. Tr{\'e}lat.
\newblock Nonlinear optimal control via occupation measures and
  {LMI} relaxations.
\newblock {SIAM Journal on Control and Optimization}, 47(4):1643--1666,
  2008.

\bibitem{lasserre2017empirical}
J. B. Lasserre, E. Pauwels.
\newblock The empirical christoffel function in statistics and machine
  learning.
\newblock {\tt arXiv:1701.02886}, 2017.

\bibitem{yalmip}
J.~L{\" o}fberg.
\newblock Yalmip : A toolbox for modeling and optimization in {Matlab}.
\newblock {Proceedings of the IEEE CACSD Conference}, Taipei, Taiwan, 2004.

\bibitem{MagronHenrionForets}
V. Magron, D. Henrion, M. Forets.
Semidefinite characterization of invariant measures for polynomial systems.
Submitted for publication, 2018.


\bibitem{meyn2012markov}
S. P. Meyn,  R. L. Tweedie.
\newblock {Markov chains and stochastic stability}.
\newblock Springer, 2012.

\bibitem{mezic2004comparison}
Mezi{\'c}, Igor, Banaszuk, Andrzej.
\newblock Comparison of systems with complex behavior.
\newblock {Physica D: Nonlinear Phenomena}, 197:101--133, 2004.

\bibitem{ozay2015set}
N. Ozay, C. Lagoa, M. Sznaier.
\newblock Set membership identification of switched linear systems with known
  number of subsystems.
\newblock {Automatica}, 51:180--191, 2015.

\bibitem{ozay2014convex}
N. Ozay, M. Sznaier, C. Lagoa.
\newblock Convex certificates for model (in) validation of switched affine
  systems with unknown switches.
\newblock {IEEE Transactions on Automatic Control}, 59(11):2921--2932,
  2014.

\bibitem{pauwels2016nips}
E. Pauwels, J. B. Lasserre.
\newblock Sorting out typicality with the inverse moment matrix sos polynomial.
\newblock {\em Advances in Neural Information Processing Systems (NIPS)}, 2016.

\bibitem{phelps2001lectures}
R. R. Phelps.
\newblock {Lectures on {C}hoquet's theorem}.
\newblock Springer, 2001.

\bibitem{putinar}
M.~Putinar.
\newblock Positive polynomials on compact semi-algebraic sets.
\newblock {Indiana University Mathematics Journal,}, 42:969--984, 1993.

\bibitem{schlosser_attractor}
C. Schlosser, M. Korda. 
\newblock{Converging outer approximations to global attractors using semidefinite programming.}
\newblock{arXiv preprint arXiv:2005.03346 (2020).}


\bibitem{sedumi}
J.~Sturm.
\newblock Using SeDuMi 1.02, a Matlab toolbox for optimization over symmetric
  cones.
\newblock {Optimization Methods and Software}, 11:625--653, 1999.

\bibitem{tobasco} I. Tobasco, D. Goluskin, C. Doering. \newblock {Optimal bounds and extremal trajectories for time averages in dynamical systems.} \newblock {APS (2017): M1-002}.




\end{thebibliography}
\end{document}